\documentclass[a4paper,12pt]{amsart}
\usepackage{hyperref,fancyhdr,mathrsfs,amsmath,amscd,amsthm,amsfonts,latexsym,amssymb,stmaryrd}
\usepackage[all]{xy}

\DeclareMathOperator{\rank}{rank}

\DeclareMathOperator{\dif}{d}

\DeclareMathOperator{\Lie}{\mathcal{L}}

\newcommand{\Cal}{\mathcal{C}}
\newcommand{\Dcal}{\mathcal{D}}
\renewcommand{\H}{\mathscr{H}}
\newcommand{\V}{\mathscr{V}}

\newcommand{\J}{\mathcal{J}}

\newcommand{\isom}{\overset{\sim}{\setbox0=\hbox{$\longrightarrow$}\ht0=0.2\ht0\box0}}

\DeclareMathOperator{\Bh}{{\it B}^{\H}}

\def \a{\alpha}

\def \b{\beta}
\def \e{\eta}
\def \ep{\varepsilon}

\def \o{\omega}

\def \phi{\varphi}
\def \Phi{\varPhi}
\def \p{\pi}
\def \r{\rho}
\def \s{\sigma}

\def \C{\mathbb{C}\,}

\def\widecheckg{g^{\hspace*{-2.5pt}\vbox to 5pt{\hbox to
0pt{\LARGE$\check{}$}}}\hspace*{2pt}}

\def\widecheckl{\lambda^{\hspace*{-3.5pt}\vbox to 8pt{\hbox to
0pt{\LARGE$\check{}$}}}\hspace*{2pt}}

\begin{document}

\title{On holomorphic maps and generalized\\
complex geometry}
\author{Liviu Ornea and Radu~Pantilie}
\thanks{The authors gratefully acknowledge that this work was partially supported by CNCSIS (Rom\^ania),
through the Project no.\ 529/06.01.2009, PN II - IDEI code~1193/2008}
\email{\href{mailto:lornea@gta.math.unibuc.ro}{lornea@gta.math.unibuc.ro},
       \href{mailto:Radu.Pantilie@imar.ro}{Radu.Pantilie@imar.ro}}
\address{L.~Ornea, Universitatea din Bucure\c sti, Facultatea de Matematic\u a, Str.\ Academiei nr.\ 14,
70109, Bucure\c sti, Rom\^ania, \emph{also,}
Institutul de Matematic\u a ``Simion~Stoilow'' al Academiei Rom\^ane,
C.P. 1-764, 014700, Bucure\c sti, Rom\^ania}
\address{R.~Pantilie, Institutul de Matematic\u a ``Simion~Stoilow'' al Academiei Rom\^ane,
C.P. 1-764, 014700, Bucure\c sti, Rom\^ania}

\newtheorem{thm}{Theorem}[section]
\newtheorem{lem}[thm]{Lemma}
\newtheorem{cor}[thm]{Corollary}
\newtheorem{prop}[thm]{Proposition}

\theoremstyle{definition}

\newtheorem{defn}[thm]{Definition}
\newtheorem{rem}[thm]{Remark}
\newtheorem{exm}[thm]{Example}

\numberwithin{equation}{section}

\maketitle
\thispagestyle{empty}
\vspace{-0.5cm}
\section*{Abstract}
\begin{quote}
{\footnotesize  We introduce a natural notion of \emph{holomorphic map} between generalized complex manifolds
and we prove some related results on Dirac structures and generalized K\"ahler manifolds.}
\end{quote}

\section*{Introduction}

\indent
The \emph{generalized complex structures} \cite{Gua-thesis}\,,\,\cite{Hit-gc_QJM} contain, as particular cases,
the complex and symplectic structures. Although for the latter structures there exist well known definitions
which give the corresponding morphisms (holomorphic maps and Poisson morphisms, respectively), it still
lacks a suitable notion of \emph{holomorphic map} with respect to which the class of generalized complex manifolds to become a category.\\
\indent
In this paper we introduce such a notion (Definition \ref{defn:ogc}\,, below) based on the
following considerations. Firstly, holomorphic maps between generalized complex manifolds should be invariant under
\emph{$B$-field transformations}. This is imposed by the fact that the group of (orthogonal) automorphisms of the Courant bracket
(which defines the integrability in Generalized Complex Geometry) on a manifold
is the semidirect product of the group of diffeomorphisms and the additive group of closed two-forms
on the manifold \cite{Gua-thesis}\,.
Secondly, by \cite{Gua-thesis}\,, underlying any linear generalized complex structure there are:\\
\indent
\quad$\bullet$ a linear Poisson structure (that is, a constant Poisson structure on the vector space; see Section \ref{section:lDs}\,), and\\
\indent
\quad$\bullet$ a linear co-CR structure (that is, a linear CR structure on the dual vector space; see Section \ref{section:gclm}\,),\\
both of which are preserved under linear $B$-field transformations. Moreover, these two structures determine,
up to a (non-unique) linear $B$-field transformation, the given generalized linear complex structure.\\
\indent
A \emph{generalized complex linear map} is a co-CR linear Poisson morphism (Definition \ref{defn:gcl}\,).
It follows quickly that a linear map is generalized complex if and only if, up to linear
$B$-fields transformations, it is the product of a (classical) complex linear map, between complex vector spaces,
and a linear Poisson morphism, between symplectic vector spaces (Proposition \ref{prop:gcl}\,).\\
\indent
A \emph{holomorphic map} between generalized (almost) complex manifolds is a map whose differential is generalized complex
(Definition \ref{defn:ogc}\,). Then, essentially, all of the above mentioned (linear) facts hold, locally,
in the setting of generalized complex manifolds (Theorem \ref{thm:local_gcs} and Proposition \ref{prop:ogc_basic}\,).\\
\indent
The first examples are the classical holomorphic maps, the Poisson morphisms between symplectic manifolds
and their products (Example \ref{exm:ogc_first}\,).\\
\indent
Other large classes of natural examples can be obtained by working with compact or nilpotent Lie groups
(Examples \ref{exm:ogc_second} and \ref{exm:ogc_third}\,).\\
\indent
Further motivation for our notion of holomorphicity comes from generalized K\"ahler geometry.
For example, if $(g,b,J_+,J_-)$ is the bi-Hermitian structure
corresponding to a generalized K\"ahler manifold $(M,L_1,L_2)$ then the holomorphic functions of $(M,L_1)$ and $(M,L_2)$ are
the bi-holomorphic functions of $(M,J_+,J_-)$ and $(M,J_+,-J_-)$\,, respectively (Remark \ref{rem:assocF_holo-functions}\,).
Other natural properties of the holomorphic maps between generalized K\"ahler manifolds are obtained in
Sections \ref{section:gKm} and \ref{section:gK_tamed} (Remark \ref{rem:Riem_subm_descend}(2) and
Corollaries \ref{cor:holo_diffeo}\,, \ref{cor:Poisson_holo_map}\,).\\
\indent
Along the way, we obtain results on generalized K\"ahler manifolds, such as the factorisation result
Theorem \ref{thm:H+-int}\,; see, also, Corollaries \ref{cor:H+-}\,, \ref{cor:VH+} and \ref{cor:L2_normal}\,, the first of which
is a significant improvement of \cite[Theorem A]{ApoGua}\,.\\
\indent
The paper is organized as follows. In Section \ref{section:lDs}\,, after recalling \cite{Courant} some
basic facts on linear Dirac structures, we give explicit descriptions (Proposition \ref{prop:pfpb}\,)
for the pull-back and push-forward of a linear Dirac structure, which we then use to show that any linear Dirac structure
is, in a natural way, the pull-back of a linear Poisson structure (Corollary \ref{cor:pbP}\,; cf.\ \cite{BurRad}\,,\,\cite{BurWei-2005}\,),
which we call the \emph{canonical (linear) Poisson quotient} (cf.\ \cite{BurWei-2005}\,), of the given linear Dirac structure. The smooth version
(Theorem \ref{thm:pbP}\,; cf.\ \cite{Courant}\,,\,\cite{BurRad}\,,\,\cite{BurWei-2005}\,) of this result is proved in
Section \ref{section:Ds} together with some other results on Dirac structures. For example, there we show (Corollary \ref{cor:blank_up_to_B}\,)
that, locally, any regular Dirac structure is, up to a $B$-field transformation, of the form $\V\oplus{\rm Ann}(\V)$\,, where $\V$ is
(the tangent bundle of) a foliation.\\
\indent
In Section \ref{section:gclm}\,, we introduce the notion of generalized complex linear map, along the above mentioned
lines. It follows that two generalized linear complex structures $L_1$ and $L_2$\,, on a vector space $V$,
can be identified if and only if $L_2$ is the linear $B$-field transform of the push-forward of $L_1$\,,
through a linear isomorphism of $V$ (Corollary \ref{cor:gcl}\,). Also, we explain (Remark \ref{rem:nonBinv}\,)
why another definition of the notion of generalized complex linear map is, in our opinion, inadequate.\\
\indent
In Section \ref{section:h_gc}\,, we review some basic facts on generalized complex manifolds and we introduce
the corresponding notion of holomorphic map. It follows that if a real analytic map $\phi$\,,
between real analytic regular generalized complex manifolds, is holomorphic then, locally,
up to the complexification of a real analytic $B$-field tranformation, the complexification of $\phi$ descends
to a complex analytic Poisson morphism between canonical Poisson quotients (Proposition \ref{prop:holo_Dirac}\,).
Also, we show that the pseudo-horizontally conformal submersions with minimal two-dimensional fibres,
from Riemannian manifolds, provide natural constructions of generalized complex structures
(Example \ref{exm:gc_harmorphs}\,).\\
\indent
In Section \ref{section:gKm}\,, we prove (Theorem \ref{thm:H+geod}\,) that if $(g,b,J_+,J_-)$ is the bi-Hermitian structure
corresponding to a generalized K\"ahler structure and we denote $\H^{\pm}={\rm ker}(J_+\mp J_-)$ then the following conditions are equivalent:\\
\indent
\quad$\bullet$ $\H^{\pm}$ integrable;\\
\indent
\quad$\bullet$ $\H^{\pm}$ geodesic.\\
It follows that, under natural conditions, the holomorphic maps between generalized K\"ahler manifolds descend to holomorphic maps
between K\"ahler manifolds (Remark \ref{rem:Riem_subm_descend}\,). Also, we classify
the generalized K\"ahler manifolds $M$ for which $TM=\H^+\oplus\H^-$ (Corollary \ref{cor:H+-}\,).\\
\indent
In Section \ref{section:gK_tamed}\,, we describe, in terms of \emph{tamed symplectic manifolds} (see Definition \ref{defn:tamed_symp}\,)
the generalized K\"ahler manifolds for which either $\H_+$ or $\H_-$ is zero; the obtained result
(Theorem \ref{thm:gK_tamed}\,) also appears, in a different form, in \cite{Gua-Pbranes}\,.
Also, in Corollary \ref{cor:VH+}\,, we prove a factorisation result for generalized K\"ahler manifolds with $\H^+$
a holomorphic foliation, with respect to $J_+$ and $J_-$\,, and $\H^-=0$ 
(or $\H^+=0$ and $\H^-$ a holomorphic foliation, with respect to $J_+$ and $J_-$); see, also,
Corollary \ref{cor:L2_normal} for a similar result and Theorem \ref{thm:H+-int} for a generalization.\\
\indent
Furthermore, we explain how the associated holomorphic Poisson structures of \cite{Hit-gc_CMP} fit into our approach
(Theorem \ref{thm:holo_Poisson_tamed}\,, Remark \ref{rem:holo_Poisson_tamed}),
we deduce some consequences for holomorphic diffeomorphisms (Corollary \ref{cor:holo_diffeo}\,),
and we show that, under natural conditions, the holomorphic maps between generalized K\"ahler manifolds
are holomorphic Poisson morphisms (Corollary \ref{cor:Poisson_holo_map}\,).

\section{Linear Dirac structures} \label{section:lDs}

\indent
In this section we recall (\,\cite{Courant}\,; see \cite{BurRad}\,,\,\cite{BurWei-2005}\,,\,\cite{Gua-thesis}\,)
some basic facts on linear Dirac structures.\\
\indent
Let $V$ be a (real or complex, finite dimensional) vector space. The symmetric bilinear
form $<\cdot,\cdot>\,$ on $V\oplus V^*$ defined by $$<u+\a,v+\b>=\tfrac12\bigl(\a(v)+\b(u)\bigr)\;,$$
for any $u+\a\,,\,v+\b\in V\oplus V^*$, corresponds, up to the factor $\tfrac12$\,, to the canonical isomorphism
$V\oplus V^*\isom\bigl(V\oplus V^*\bigr)^*$ defined by $u+\a\longmapsto\a+u$\,, for any
$u+\a\in V\oplus V^*$\,. In particular, $<\cdot,\cdot>$ is nondegenerate and, if $V$ is real,
its index is $\dim V$\,. Thus, the dimension of the maximal isotropic subspaces of $V\oplus V^*$
(endowed with $<\cdot,\cdot>$) is equal to $\dim V$.

\begin{defn}[\,\cite{Courant}\,]
A \emph{linear Dirac structure} on $V$ is a maximal isotropic subspace of $V\oplus V^*$.
\end{defn}

\indent
If $b$ is a bilinear form on $V$ then we shall denote by the same letter the corresponding
linear map from $V$ to $V^*$; thus, $b(u)(v)=b(u,v)$\,, for any $u\,,v\in V$.\\
\indent
Let $E\subseteq V$ be a vector subspace and let $\ep\in\Lambda^2E^*$; denote
$$L(E,\ep)=\bigl\{\,u+\a\,\big|\,u\in E\,,\,\a|_E=\ep(u)\,\bigr\}\;.$$
From the fact that $\ep$ is skew-symmetric it follows easily that $L(E,\ep)$ is isotropic.
Also, $L(E,0)=E\oplus{\rm Ann}(E)$\,, where ${\rm Ann}(E)=\bigl\{\a\in V^*\big|\,\a|_E=0\,\bigr\}$\,.\\\\
\indent
We shall denote by $\p$ and $^{*\!}\p$ the projections from $V\oplus V^*$ onto $V$ and $V^*$, respectively.
Also, if $L\subseteq V\oplus V^*$ then $L^{\perp}$ denotes the `orthogonal complement' of $L$ with respect to $<\cdot,\cdot>$\,.

\begin{prop}[\,\cite{Courant}\,] \label{prop:L}
Let $L$ be an isotropic subspace of $V\oplus V^*$ and let $E=\p(L)$\,.\\
\indent
Then there exists a unique $\ep\in\Lambda^2E^*$ such that $L\subseteq L(E,\ep)$\,.
In particular, if $L$ is a linear Dirac structure then $L=L(E,\ep)$\,. Furthermore,
$V\cap L={\rm ker}\,\ep$ and $^{*\!}\p(L)={\rm Ann}(V\cap L)$\,.
\end{prop}

\indent
Let $L$ be a linear Dirac structure on $V$. If\/ $^{*\!}\p(L)=V^*$ then $L$ is called a
\emph{linear Poisson structure} (see \cite{Courant}\,). By Proposition \ref{prop:L}\,,
if $L$ is a linear Poisson structure then $L=L(V^*,\eta)$ for some \emph{bivector} $\eta\in\Lambda^2V$
(cf.\ \cite{Wei-local_P}\,).\\

\indent
Let $V$ and $W$ be vector spaces endowed with linear Dirac structures $L_V$ and $L_W$,
respectively, and let $f:V\to W$ be a linear map. Denote
\begin{equation*}
\begin{split}
f_*(L_V)&=\bigl\{f(X)+\e\,|\,X+f^*(\e)\in L_V\bigr\}\;,\\
f^*(L_W)&=\bigl\{X+f^*(\e)\,|\,f(X)+\e\in L_W\bigr\}\;.
\end{split}
\end{equation*}

\begin{prop} \label{prop:pfpb}
Let $f:V\to W$ be a linear map. Let $L(E,\ep)$ and $L(F,\e)$ be linear Dirac structures
on $V$ and $W$, respectively. Then
\begin{equation*}
\begin{split}
f_*\bigl(L(E,\ep)\bigr)&=L\bigl(f\bigl((E\cap\ker\!f)^{\perp_{\ep}}\bigr),\check{\ep}\,\bigr)\;,\\
f^*\bigl(L(F,\e)\bigr)&=L\bigl(f^{-1}(F),f^*(\e)\bigr)\;,
\end{split}
\end{equation*}
where $\check{\ep}$ is characterised by $f^*(\check{\ep})=\ep$ on
$(E\cap\ker\!f )^{\perp_{\ep}}$.
\end{prop}
\begin{proof}
It is easy to prove that $f_*(L_V)$ and $f^*(L_W)$ are isotropic subspaces of
$W\oplus W^*$ and $V\oplus V^*$, respectively.\\
\indent
Next, we show that there exists a unique two-form $\check{\ep}$ on
$f\bigl((E\cap\ker\!f )^{\perp_{\ep}}\bigr)$ such that $f^*(\check{\ep})=\ep$ on
$(E\cap\ker\!f )^{\perp_{\ep}}$. For this, it is sufficient to prove that if
$X_1,X_2\in (E\cap\ker\!f )^{\perp_{\ep}}$ are such that $f(X_1)=f(X_2)$
then $\ep(X_1,Y)=\ep(X_2,Y)$\,, for any $Y\in(E\cap\ker\!f )^{\perp_{\ep}}$.
Now, if $X_1,X_2\in (E\cap\ker\!f )^{\perp_{\ep}}$, then $X_1,X_2\in E$
and, as $X_1-X_2\in\ker\!f $, we have $\ep(X_1-X_2,Y)=0$\,, for any
$Y\in(E\cap\ker\!f )^{\perp_{\ep}}$.\\
\indent
Thus, to complete the proof it is sufficient to show that
\begin{equation} \label{e:pfpb1}
\begin{split}
f_*\bigl(L(E,\ep)\bigr)&\supseteq L\bigl(f\bigl((E\cap\ker\!f )^{\perp_{\ep}}\bigr),\check{\ep}\,\bigr)\;,\\
f^*\bigl(L(F,\e)\bigr)&\supseteq L\bigl(f^{-1}(F),f^*(\e)\bigr)\;.
\end{split}
\end{equation}
\indent
Let $Y+\xi\in L\bigl(f\bigl((E\cap\ker\!f )^{\perp_{\ep}}\bigr),\check{\ep}\,\bigr)$\,;
equivalently, there exists $X\in(E\cap\ker\!f )^{\perp_{\ep}}$ such that $f(X)=Y$
and $\xi(f(X'))=\ep(f(X),f(X'))$\,, for any $X'\in(E\cap\ker\!f )^{\perp_{\ep}}$.\\
\indent
We claim that $Y+\xi\in f_*\bigl(L(E,\ep)\bigr)$\,; equivalently, there exists
$X\in(E\cap\ker\!f )^{\perp_{\ep}}$ such that $f(X)=Y$ and $\xi(f(X'))=\ep(X,X')$\,,
for any $X'\in E$\,.\\
\indent
It is easy to prove that, if $X\in(E\cap\ker\!f )^{\perp_{\ep}}$ is such that $f(X)=Y$,
then $\xi(f(X'))=\ep(X,X')$\,, for any
$X'\in(E\cap\ker\!f )\cup(E\cap\ker\!f )^{\perp_{\ep}}$\,.\\
\indent
It follows that, for any $X\in(E\cap\ker\!f )^{\perp_{\ep}}$ with $f(X)=Y$, there exists
$X_1\in\ker(\ep|_{E\cap\ker\!f })$ such that
$\xi(f(X'))=\ep(X+X_1,X')$\,, for any $X'\in E$\,; as, then, we also have
$X_1\in(E\cap\ker\!f )^{\perp_{\ep}}$ and $f(X_1)=0$,
this shows that $Y+\xi=f(X+X_1)+\xi\in f_*(L_V)$\,.\\
\indent
To prove the second relation of \eqref{e:pfpb1}\,, let
$X+\xi\in L\bigl(f^{-1}(F),f^*(\e)\bigr)$\,; equivalently, $f(X)\in F$ and
$\xi(X')=\e(f(X),f(X'))$ for any $X'\in f^{-1}(F)$\,. As $f^{-1}(F)\supseteq\ker\!f $,
there exists $\check{\xi}$ in the dual of $f(V)$ such that $\xi=f^*(\check{\xi})$\,.
Obviously, we can extend $\check{\xi}$ to an one-form on $W$, which we shall denote
by the same symbol $\check{\xi}$, such that $\check{\xi}(Y)=\e(f(X),Y)$\,,
for any $Y\in F$; equivalently, $f(X)+\check{\xi}\in L(F,\e)$\,. Therefore
$X+\xi=X+f^*(\check{\xi})\in f^*\bigl(L(F,\e)\bigr)$.\\
\indent
The proof is complete.
\end{proof}

\begin{defn}[see \cite{BurRad}\,,\,\cite{BurWei-2005}\,,\,\cite{Gua-thesis}\,]
Let $V$ and $W$ be vector spaces endowed with linear Dirac structures $L_V$ and $L_W$\,,
respectively, and let $f:V\to W$ be a linear map.\\
\indent
Then $f_*(L_V)$ and $f^*(L_W)$ are called the \emph{push forward \emph{and} pull back,
by $f$, of $L_V$ \emph{and} $L_W$}\,,
respectively.
\end{defn}

\indent
Note that, if $f:(V,L_V)\to(W,L_W)$ is a linear map between vector spaces endowed with linear
Poisson structures then the following assertions are equivalent (see \cite{BurRad}\,,\,\cite{BurWei-2005}\,):\\
\indent
(i) $f$ is a \emph{linear Poisson morphism} (that is, $f(\eta_{\,V})=\eta_{\,W}$, where $\eta_{\,V}$ and $\eta_{\,W}$
are the bivectors defining $L_V$ and $L_W$, respectively; see \cite{Vai-Poisson_book}\,).\\
\indent
(ii) $f_*(L_V)=L_W$.\\

\indent
From Proposition \ref{prop:pfpb}\,, we easily obtain the following result.

\begin{cor}[cf.\ \cite{BurRad}\,,\,\cite{BurWei-2005}\,] \label{cor:pbP}
Let $V$ be a vector space endowed with a linear Dirac structure $L=L(E,\ep)$\,.
Let $W=\ker\ep$ and denote by $\phi:V\to V/W$ the projection.\\
\indent
Then $L=\phi^*(\phi_*(L))$ and $\phi_*(L)$ is a linear Poisson structure on $V/W$.
\end{cor}

\section{Dirac structures} \label{section:Ds}

\indent
In this section, we shall work in the smooth and (real or complex) analytic categories.
All the notations of Section \ref{section:lDs} will be applied to tangent bundles of manifolds and
to (differentials of) maps between manifolds.

\begin{defn}[\,\cite{Courant}\,]
An \emph{almost Dirac structure} on a manifold $M$ is a maximal isotropic subbundle of $TM\oplus T^*M$.\\
\indent
An almost Dirac structure is \emph{integrable} if it's space of sections is closed under the \emph{Courant bracket}
defined by
$$[X+\a,Y+\b]=[X,Y]+\tfrac12\dif\bigl(\iota_X\b-\iota_Y\a\bigr)+\iota_{X\!}\dif\!\b-\iota_{Y\!}\dif\!\a\,,$$
for any sections $X+\a$ and $Y+\b$ of $TM\oplus\Lambda(T^*M)$\,, where $\iota$ denotes the interior product.\\
\indent
A \emph{Dirac structure} is an integrable almost Dirac structure.
\end{defn}

\indent
Let $L$ be a Dirac structure on $M$. If\/ $\p(L)=TM$ then $L$ is a \emph{presymplectic structure} whilst if\/
$^{*\!}\p(L)=T^*M$ then $L$ is a \emph{Poisson structure} \cite{Courant} (cf.\ \cite{Wei-local_P}\,).\\
\indent
Recall \cite[\S4]{Courant} that a point of a manifold endowed with an almost Dirac structure $L$ is called \emph{regular}
if, in some open neighbourhood of it, $\p(L)$ and $^{*\!}\p(L)$ are bundles.\\
\indent
The following result follows from the fact that it is sufficient to be proved for maps of constant rank
between manifolds endowed with regular almost Dirac structures.

\begin{prop} \label{prop:funct_integr}
Let $M$ and $N$ be manifolds endowed with the almost Dirac structures $L_M$ and $L_N$, respectively. Let
$\phi:M\to N$ be a map which maps regular points of $L_M$ to regular points of $L_N$\,.\\
\indent
{\rm (i)} If $L_M$ is integrable and $\phi_*(L_M)=L_N$ then $L_N$ is integrable.\\
\indent
{\rm (ii)} If $L_N$ is integrable and $\phi^*(L_N)=L_M$ then $L_M$ is integrable.
\end{prop}

\indent
Next, we prove the following result.

\begin{thm}[cf.\ \cite{Courant}\,,\,\cite{BurRad}\,,\,\cite{BurWei-2005}\,] \label{thm:pbP}
Let $L$ be a Dirac structure on $M$ such that $^{*\!}\p(L)$ is a subbundle of $T^*M$. Then, locally, there exist
submersions $\phi$ on $M$ such that $\phi_*(L)$ is a Poisson structure and $L=\phi^*(\phi_*(L))$\,; moreover,
these submersions are (germ) unique, up to Poisson diffeomorphisms of their codomains.
\end{thm}
\begin{proof}
By hypothesis, $TM\cap L$ is a subbundle of $TM$. Furthermore, as $L$ is integrable, $TM\cap L$ is
(the tangent bundle to) a foliation.\\
\indent
Let $F={}^{*\!}\p(L)$ and let $\eta$ be the section of $\Lambda^2F^*$ such that $L=L(F,\eta)$\,. Note that,
$F\bigl(={\rm Ann}(TM\cap L)\bigr)$ is locally spanned by the differentials of functions which are basic
with respect to $TM\cap L$\,.\\
\indent
Let $f$ and $g$ be functions, locally defined on $M$, such that $\dif\!f$ and $\dif\!g$ are sections of $F$.
Then there exists vector fields $X$ and $Y$, locally defined on $M$, such that $X+\dif\!f$ and $Y+\dif\!g$
are local sections of $L$\,; in particular, we have $\eta(\dif\!f,\dif\!g)=X(g)=-Y(f)$\,.
Hence $[X+\dif\!f,Y+\dif\!g]=[X,Y]+\dif\bigl(\eta(\dif\!f,\dif\!g)\bigr)$ and we deduce that $\eta(\dif\!f,\dif\!g)$
is basic with respect to $TM\cap L$\,.\\
\indent
The proof follows quickly from Corollary \ref{cor:pbP} and Proposition \ref{prop:funct_integr}\,.
\end{proof}

\indent
Under the same hypotheses, as in Theorem \ref{thm:pbP}\,, we call $\phi_*(L)$ the \emph{canonical (local) Poisson quotient}
of $L$\,.\\
\indent
Next, we prove the following (cf.\ \cite[Proposition 4.1.2]{Courant}\,).

\begin{prop} \label{prop:pf_to_presym}
Let $L=L(E,\ep)$ be a Dirac structure on $M$ and let $x\in M$ be a regular point of $L$; denote by $P$ the leaf of $E$ through $x$.\\
\indent
Then for any submanifold $Q$ of $M$ transversal to $E$, such that $x\in Q$ and $\dim Q=\dim M-\dim P$, there exists a submersion $\r$ from
some open neighbourhood $U$\! of $x$ in $M$ onto some open neighbourhood $V$\! of $x$ in $P$ such that $\r_*(L|_U)=L(TV,\ep|_V)$
and the fibre of $\r$ through $x$ is an open set of $Q$\,.
\end{prop}
\begin{proof}
From Theorem \ref{thm:pbP} it follows that we may assume $L$ a Poisson structure.\\
\indent
If we ignore the fact that the fibre of $\r$ through $x$ is fixed then the proposition is a consequence of
\cite[Corollary 2.3]{Wei-local_P} and Proposition \ref{prop:pfpb}\,.
To complete the proof just note that in the proof of \cite[Theorem 2.1]{Wei-local_P} (and, consequently, of
\cite[Corollary 2.3]{Wei-local_P}\,, as well), at each step, the two functions involved may be assumed
constant along $Q$.
\end{proof}

\indent
Recall (see \cite{Gua-thesis}\,,\,\cite{BurRad}\,) that any closed two-form $B$ on $M$ corresponds to a \emph{$B$-field transformation}
which is the automorphism of $TM\oplus T^*M$, preserving the Courant bracket, defined by
$${\rm exp}(B)(X+\a)=X+B(X)+\a\;$$
for any $X+\a\in TM\oplus T^*M$, where, as before, we have identified $B$ with the corresponding section of ${\rm Hom}(TM,T^*M)$\,.
It is easy to prove that if $L=L(E,\ep)$ is an almost Dirac structure on $M$ then ${\rm exp}(B)(L)=L(E,\ep+B|_E)$\,.

\begin{cor} \label{cor:blank_up_to_B}
Let $L$ be a regular Dirac structure on $M$; denote $E=\p(L)$\,. Then, locally, there exist two-forms $B$ on $M$
such that ${\rm exp}(B)(L)=E\oplus{\rm Ann}E$.
\end{cor}
\begin{proof}
By Proposition \ref{prop:pf_to_presym}\,, locally, there exist submersions $\r:M\to P$ onto presymplectic manifolds
$\bigl(P,L(TP,\o)\bigr)$ such that $\r_*(L)=L(TP,\o)$\,.\\
\indent
Then $B=-\r^*(\o)$ is as required.
\end{proof}

\indent
We end this section with the following result which will be used later on.

\begin{prop} \label{prop:Poisson_OK}
Let $\phi:(M,L_M)\to(N,L_N)$ be a Poisson morphism, of constant rank, between regular Poisson manifolds
such that $\dif\!\phi(E_M)\subseteq E_N$\,, where $E_M$ and $E_N$ are the (symplectic) foliations determined by $L_M$
and $L_N$\,, respectively.\\
\indent
Then, locally, there exist submersions $\r:M\to(P,\o)$ and $\s:N\to(Q,\e)$ onto symplectic manifolds,
and a Poisson morphism $\psi:(P,\o)\to(Q,\e)$ such that:\\
\indent
\quad{\rm (i)} $TM=E_M\oplus{\rm ker}\dif\!\r$ and $\r_*(L_M)=L(TP,\o)$\,;\\
\indent
\quad{\rm (ii)} $TN=E_N\oplus{\rm ker}\dif\!\s$ and $\s_*(L_N)=L(TQ,\e)$\,;\\
\indent
\quad{\rm (iii)} $\s\circ\phi=\psi\circ\r$\,.
\end{prop}
\begin{proof}
From Proposition \ref{prop:pfpb} we obtain that $\dif\!\phi(E_M)=E_N$\,. As, locally, $\phi$ is the composition
of a submersion followed by an immersion, it follows that we may assume that $\phi$ is a surjective submersion.\\
\indent
By Proposition \ref{prop:pf_to_presym}\,, locally, there exists a submersion $\s:M\to(Q,\e)$
onto a symplectic manifold such that assertion (ii) is satisfied.\\
\indent
Let $\V$ be the distribution on $M$ generated by all of the Hamiltonian vector fields determined by $u\circ\s\circ\phi$\,,
with $u$ a function on $Q$\,; obviously, $\V\subseteq E_M$\,.
Then arguments similar to the inductive step of the proof of \cite[Theorem 2.1]{Wei-local_P} show that:\\
\indent
\quad(a) $\V$ is a foliation mapped by $\s\circ\phi$ onto $TQ$\,;\\
\indent
\quad(b) $\V$ and $E_M\cap{\rm ker}\dif\!\phi$ are nondegenerate and complementary orthogonal
with respect to the symplectic structure $\o_M$ of $E_M$\,;\\
\indent
\quad(c) $\o_M$ restricted to $\V$ is projectable (onto $\e$\,) with respect to $\s\circ\phi$\,;\\
\indent
\quad(d) $\o_M$ restricted to $E_M\cap{\rm ker}\dif\!\phi$ is projectable with respect to $\V$.\\
\indent
Consequently, $(E_M,\o_M)$ induces on any fibre $M'$ of $\s\circ\phi$ a Poisson structure $L'$ such that, locally,
$(M,L_M)$ is the product of $(M',L')$ and $\bigl(Q,L(TQ,\e)\bigr)$\,.\\
\indent
By Proposition \ref{prop:pf_to_presym}\,, locally, there exists a submersion $\r':M'\to(P',\o')$ such that
${\rm ker}\dif\!\r'\oplus(E_M\cap TM')=TM'$ and $\r'_*(L')=L(TP',\o')$\,.\\
\indent
If we define $(P,\o)=(P',\o')\times(Q,\e)$\,, $\r=\r'\times\s$ and $\psi:P\to Q$ the projection
then it is easy to see that $\r$\,, $\s$ and $\psi$ are as required.
\end{proof}

\section{Generalized complex linear maps} \label{section:gclm}

\indent
A \emph{linear generalized complex structure} on a vector space $V$ is a maximal isotropic subspace $L=L(E,\ep)$ of
$V^{\C}\oplus\bigl(V^{\C}\bigr)^*$ such that $L\cap\overline{L}=\{0\}$\ \cite{Gua-thesis}\,,\,\cite{Hit-gc_QJM}\,;
equivalently, $E+\overline{E}=V^{\C}$ and ${\rm Im}\bigl(\ep|_{E\cap\overline{E}}\bigr)$ is
nondegenerate \cite{Gua-thesis}\,.\\
\indent
The condition $E+\overline{E}=V^{\C}$ means that $E$ is a \emph{linear co-CR structure} on $V$ \cite{fq}\,;
equivalently, the annihilator $E^0$ of $E$ is a linear CR structure on $V$ (that is, $E^0\cap\overline{E^0}=\{0\}$).\\
\indent
On the other hand, as ${\rm Im}\bigl(\ep|_{E\cap\overline{E}}\bigr)$ is nondegenerate,
$L\bigl(E\cap\overline{E},{\rm Im}\bigl(\ep|_{E\cap\overline{E}}\bigr)\bigr)$ is a linear Poisson structure
on $V$.\\
\indent
If $L=L(E,\ep)$ is a linear generalized complex structure then we call $E$ and
$L\bigl(E\cap\overline{E},{\rm Im}\bigl(\ep|_{E\cap\overline{E}}\bigr)\bigr)$
\emph{the associated linear co-CR} and \emph{Poisson structures}, respectively.\\
\indent
A map $f:(V,E_V)\to(W,E_W)$ between vector spaces endowed with linear (co-)CR structures is
a \emph{(co-)CR linear map} if it is linear and $f(E_V)\subseteq E_W$\,.

\begin{defn} \label{defn:gcl}
A linear map between vector spaces endowed with linear generalized complex structures
is \emph{generalized complex linear} if it is a co-CR linear Poisson morphism,
with respect to the associated linear co-CR and Poisson structures.
\end{defn}

\indent
Note that, Definition \ref{defn:gcl} is invariant under linear $B$-field transformations.
Also, the composition of two generalized complex linear maps is a generalized complex linear map.

\begin{prop} \label{prop:gcl}
Let $f:V\to W$ be a linear map between vector spaces endowed with linear generalized complex
structures $L_V$\! and $L_W$, respectively.\\
\indent
Then the following assertions are equivalent:\\
\indent
\quad{\rm (i)} $f$ is generalized complex linear.\\
\indent
\quad{\rm (ii)} Up to linear $B$-field transformations, $f$ is the direct sum of a a complex linear map, between complex vector spaces,
and a linear Poisson morphism, between symplectic vector spaces.
\end{prop}
\begin{proof}
Suppose that (i) holds and let $E_V$ and $E_W$ be the linear co-CR structures associated to $L_V$ and $L_W$\,, respectively.\\
\indent
As $f:(V,E_V)\to(W,E_W)$ is co-CR linear, we obtain $f\bigl(E_V\cap\overline{E_V}\bigr)\subseteq E_W\cap\overline{E_W}$\,.
But $f$ is, also, a linear Poisson morphism, with respect to the linear Poisson structures associated to $L_V$ and $L_W$\,,
respectively. From Proposition \ref{prop:pfpb} we obtain that $f\bigl(E_V\cap\overline{E_V}\bigr)=E_W\cap\overline{E_W}$.
Moreover, $f$ restricts to give a linear Poisson morphism between $E_V\cap\overline{E_V}$ and $E_W\cap\overline{E_W}$\,,
endowed with the linear symplectic structures corresponding to the linear Poisson structures associated to $L_V$ and $L_W$\,,
respectively.\\
\indent
We, also, obtain $f^{-1}\bigl(E_W\cap\overline{E_W}\bigr)={\rm ker}\,f+\bigl(E_V\cap\overline{E_V}\bigr)$
and, consequently, there exist complementary vector spaces $V'$ and $W'$ of $E_V\cap\overline{E_V}$
and $E_W\cap\overline{E_W}$ in $V$ and $W$, respectively, such that $f\bigl(V'\bigr)\subseteq W'$.\\
\indent
It is obvious that $E_V$ and $E_W$ induce linear complex structures on $V'$ and $W'$, respectively.
Moreover, $f$ restricts to give a complex linear map between these two complex vector spaces.\\
\indent
Now, (i)$\Longrightarrow$(ii) follows quickly from \cite[Theorem 4.13]{Gua-thesis}\,, whilst (ii)$\Longrightarrow$(i) is trivial.
\end{proof}

\indent
The next result is an immediate consequence of Proposition \ref{prop:gcl}\,.

\begin{cor} \label{cor:gcl}
Let $f:V\to W$ be a linear isomorphism between vector spaces endowed with linear generalized complex
structures $L_V$\! and $L_W$, respectively.\\
\indent
Then the following assertions are equivalent:\\
\indent
\quad{\rm (i)} $f$ is generalized complex linear.\\
\indent
\quad{\rm (ii)} $f_*\bigl(L_V\bigr)=L_W$, up to linear $B$-field transformations.
\end{cor}

\indent
We end this section with the following:

\begin{rem} \label{rem:nonBinv}
It has been proposed another definition for the notion of generalized complex linear map by imposing that the
product of the graphs of the map and of its transpose be invariant under the product of the
(endomorphisms corresponding to the) generalized linear complex structures, of the domain and codomain \cite{Cra} (see \cite{Vai-red_gc}\,).\\
\indent
However this notion is not invariant under linear $B$-field transformations as we shall now explain.\\
\indent
Let $(V,J)$ be a complex vector space and let $b$ be a two-form on $V$; denote by $L_J$ the linear generalized complex structure
corresponding to $J$. Then the map ${\rm Id}_V:\bigl(V,L_J\bigr)\to\bigl(V,L_{({\rm exp}\,b)(L_J)}\bigr)$ satisfies
the above mentioned condition if and only if $b$ is of type $(1,1)$\,, with respect to $J$.\\
\indent
Certainly, this inconvenience would be removed if we take this definition up to linear $B$-field transformations.
However, a straightforward calculation shows that there are no such maps between symplectic vector spaces $U$ and $V$
with $\dim U-\dim V=2$\,, a rather unnatural restriction.
\end{rem}

\section{Holomorphic maps between generalized complex manifolds} \label{section:h_gc}

\indent
{}From now on, unless otherwise stated, all the manifolds are assumed connected and smooth and all the maps are
assumed smooth.\\
\indent
A \emph{generalized almost complex structure} on $M$ is a complex vector subbundle $L$ of
$T^{\C\!}M\oplus\bigl(T^{\C\!}M\bigr)^*$ such that $L_x$ is a linear generalized complex structure
on $T_xM$, for any $x\in M$. An \emph{integrable} generalized complex structure is a generalized
almost complex structure whose space of sections is closed under the (complexification of the)
Courant bracket; a \emph{generalized (almost) complex manifold} is a manifold endowed with a generalized
(almost) complex structure \cite{Gua-thesis}\,,\,\cite{Hit-gc_QJM}\,.

\begin{defn} \label{defn:ogc}
A map between generalized almost complex manifolds is \emph{holomorphic} if, at each point,
its differential is generalized complex linear.
\end{defn}

\indent
A point $x$ of a generalized almost complex manifold $(M,L)$ is \emph{regular} if it is regular
for the associated almost Poisson structure; equivalently, in some open neighbourhood of $x$\,,
$\p(L)$ is a complex vector subbundle of $T^{\C\!}M$.\\
\indent
An \emph{almost $($co-$)$CR structure} on a manifold $M$ is a complex vector subbundle $\Cal$ of $T^{\C}\!M$
such that $\Cal_x$ is a linear (co-)CR structure on $T_xM$, for any $x\in M$. An
\emph{integrable} almost (co-)CR structure is an almost (co-)CR structure whose space of sections is
closed under the (Lie) bracket; a \emph{$($co-$)$CR structure} is an integrable almost (co-)CR structure (see \cite{fq}\,).\\
\indent
Note that, the eigenbundles of a complex structure are both CR and co-CR structures.
Also, a generalized almost complex structure $L$ on $M$ is regular (at each point) if and only if $\p(L)$
is an almost co-CR structure on $M$.\\
\indent
Let $\phi:M\to N$ be a submersion onto a complex manifold $(N,J)$\,; denote by $T^{1,0}N$
the eigenbundle of $J$ corresponding to ${\rm i}$\,.  Then $\dif\!\phi^{-1}\bigl(T^{1,0}N\bigr)$
is a co-CR structure on $M$. Conversely, any co-CR structure is, locally, obtained this way.\\
\indent
A map between manifolds endowed with almost (co-)CR structures is \emph{$($co-$)$CR holomorphic}
if, at each point, its differential is a \mbox{(co-)CR} linear map.\\
\indent
An \emph{almost $f$-structure} is a $(1,1)$-tensor field $F$ such that $F^3+F=0$\,. Any almost
$f$-structure on $M$ corresponds to a pair formed of an almost CR structure $\Cal$ and
an almost co-CR structure $\Dcal$\,, which are compatible \cite{fq}\,; these are given by
$\Cal=T^{1,0}M$ and $\Dcal=T^0M\oplus T^{1,0}M$, where
$T^0M$ and $T^{1,0}M$ are the eigenbundles of $F$ corresponding to $0$ and ${\rm i}$\,, respectively.\\
\indent
An almost $f$-structure is \emph{(co-)CR integrable} if the associated almost (co-)CR structure
is integrable. An \emph{(integrable almost) $f$-structure} is an almost $f$-structure which is
both CR and co-CR integrable \cite{fq}\,.\\
\indent
An almost $f$-structure $F$ and a two form $\o$ on $M$ are \emph{compatible} if $\o$ is nondegenerate on
$T^0M$ and $\iota_X\o=0$\,, for any $X\in T^{1,0}M\oplus T^{0,1}M$.\\
\indent
A generalized (almost) complex structure $L$ on $M$ is in \emph{normal form} if
$L=L\bigl(T^0M\oplus T^{1,0}M,{\rm i}\,\o\bigr)$ for some compatible almost $f$-structure and two-form
$\o$ on $M$. Note that, a generalized almost complex structure in normal form is regular.

\begin{prop} \label{prop:int_normal_form}
Let $L=L\bigl(T^0M\oplus T^{1,0}M,{\rm i}\,\o\bigr)$ be the generalized almost complex structure
in normal form, corresponding to the compatible almost $f$-structure $F$ and two-form $\o$ on $M$.\\
\indent
Then the following assertions are equivalent:\\
\indent
\quad{\rm (i)} $L$ is integrable.\\
\indent
\quad{\rm (ii)} $F$ is integrable, $L(T^0M,\o)$ is a Poisson structure and $\o$ is invariant
under the parallel displacement of $T^{1,0}M\oplus T^{0,1}M$.
\end{prop}
\begin{proof}
From \cite[Proposition 4.19]{Gua-thesis} it follows quickly that assertion (i) is equivalent to the fact that
$F$ is co-CR integrable and $(\dif\!\o)|_{T^0M\oplus T^{1,0}M}=0$. Assuming $F$ co-CR integrable, the latter
condition is equivalent to the fact that $L(T^0M,\o)$ is a Poisson structure, $F$ is CR integrable and
$(\mathcal{L}_X\o)|_{T^0M}=0$ for any vector field $X$ tangent to $T^{1,0}M\oplus T^{0,1}M$,
where $\mathcal{L}$ denotes the Lie derivative.
\end{proof}

\indent
All of the examples of generalized complex structures of \cite{CavGua-nil} are in normal form.
Similarly, we have the following example, due to \cite{AleDav-GC}\,.

\begin{exm} \label{exm:normal_GC_on_Lie_groups}
Let $G$ be a compact Lie group of even rank assumed, for simplicity, semisimple. Let $\mathfrak{g}$ be the Lie algebra of $G$
and let $\mathfrak{k}$ be the Lie algebra of a maximal torus in $G$.\\
\indent
Let $\mathfrak{c}$ be a Borel subalgebra of $\mathfrak{g}^{\C}$
containing $\mathfrak{k}^{\C}$. Any such Borel subalgebra is obtained by choosing a base for the root system of $\mathfrak{g}^{\C}$
corresponding to $\mathfrak{k}^{\C}$ (see \cite{Hum}\,):
$\mathfrak{c}=\mathfrak{k}^{\C}\oplus\bigoplus_{\a\succ0}\mathfrak{g}^{\a}$,
where $\mathfrak{g}^{\a}$ is the root space of $\mathfrak{g}^{\C}$ corresponding to the root $\a$.\\
\indent
As $\overline{\mathfrak{g}^{\a}}=\mathfrak{g}^{-\a}$ (see \cite{BurRaw}\,),
we have $\mathfrak{c}+\overline{\mathfrak{c}}=\mathfrak{g}^{\C}$ and $\mathfrak{c}\cap\overline{\mathfrak{c}}=\mathfrak{k}^{\C}$.
Consequently, $\mathfrak{c}$ corresponds to a left invariant co-CR structure $\Cal$ on $G$ (for any $a\in G$, we have that
$\Cal_a$ is the left translation of $\mathfrak{c}$\,, at $a$).\\
\indent
Let $\o$ be a linear symplectic form on $\mathfrak{k}$ ($\dim\mathfrak{k}=\rank G$ is even), extended to $\mathfrak{g}$ such that
$\iota_X\o=0$ for any $X\in\bigoplus_{\a\succ0}\mathfrak{g}^{\a}$.
We shall denote by the same letter $\o$ the left invariant two-form on $G$, determined by $\o$.\\
\indent
Then $L\bigl(\Cal,{\rm i}\,\o\bigr)$ is a generalized complex structure on $G$ in normal form.
\end{exm}

\indent
The next result follows from the proof of \cite[Theorem 4.35]{Gua-thesis}\,.

\begin{thm} \label{thm:local_gcs}
Let $L$ be a regular generalized almost complex structure on $M$ and let $L'$ be the associated almost Poisson structure.\\
\indent
Then the following assertions are equivalent:\\
\indent
\quad{\rm (i)} $L$ is integrable.\\
\indent
\quad{\rm (ii)} $\p(L)$ and $L'$ are integrable and, locally, for any submersion $\r:M\to P$,
with $\dim P=\rank\bigl(\p(L')\bigr)$ and $\r_*(L')$ a symplectic structure on $P$,
we have that, up to a $B$-field transformation, $L$ is in normal form
with respect to the $f$-structure on $M$ determined by $\p(L)$ and $\p(L)\cap{\rm ker}\dif\!\r$.
\end{thm}

\indent
We can, now, give the smooth version of Proposition \ref{prop:gcl}\,.

\begin{prop} \label{prop:ogc_basic}
Let $\phi:(M,L_M)\to(N,L_N)$ be a map between generalized complex manifolds.\\
\indent
Then the following assertions are equivalent:\\
\indent
\quad{\rm (i)} $\phi$ is holomorphic.\\
\indent
\quad{\rm (ii)} On an open neighbourhood of each regular point of $L_M$ on which $\phi$ has constant rank,
up to $B$-field transformations, $\phi$ is the product of a Poisson morphism between symplectic manifolds
and a holomorphic map between complex manifolds.
\end{prop}
\begin{proof}
This follows quickly from Proposition \ref{prop:Poisson_OK}\,, \cite[Theorem 4.35]{Gua-thesis}
and Theorem \ref{thm:local_gcs}\,.
\end{proof}

\indent
Next, we give examples of holomorphic maps between generalized complex manifolds.

\begin{exm} \label{exm:ogc_first}
The classical holomorphic maps, the Poisson morphisms between symplectic manifolds, and their products
are, obviously, holomorphic maps between generalized complex manifolds.\\
\indent
Moreover, by Proposition \ref{prop:ogc_basic}\,, any holomorphic map $\phi:(M,L_M)\to(N,L_N)$
between generalized complex manifolds is of this form, up to $B$-field transformations, on an open neighbourhood
of each regular point of $L_M$ on which $\phi$ has constant rank.
\end{exm}

\begin{exm} \label{exm:ogc_second}
Let $G$ be a compact Lie group endowed with the generalized complex structures
$L=L\bigl(\Cal,{\rm i}\,\o\bigr)$ of Example \ref{exm:normal_GC_on_Lie_groups}\,.\\
\indent
Let $K$ be the maximal torus of $G$ whose Lie algebra is used to define $\Cal$.
Obviously, $\dif\!\phi(\Cal)$ defines a left invariant complex structure on $G/K$\,, where $\phi:G\to G/K$ is the projection.\\
\indent
Then $\phi:(G,L)\to\bigl(G/K,\dif\!\phi(\Cal)\bigr)$ is a holomorphic map.
\end{exm}

\begin{exm} \label{exm:ogc_third}
Let $G/H$ be a compact inner symmetric space (see \cite[page 23]{BurRaw} for the definition and \cite[page 38]{BurRaw}
for a table of examples) with $\rank G(=\rank H)$ even; denote by $\mathfrak{g}$ and $\mathfrak{h}$ the Lie algebras
of $G$ and $H$, respectively.\\
\indent
Endow $G$ with the generalized complex structures $L\bigl(\Cal,{\rm i}\,\o\bigr)$ of Example \ref{exm:normal_GC_on_Lie_groups}\,,
determined by a Borel subalgebra $\mathfrak{c}$ of $\mathfrak{g}^{\C}$ containing the Lie algebra of a maximal torus of $H$
(also a maximal torus of $G$, as $G/H$ is inner).\\
\indent
It follows that $\mathfrak{d}=\mathfrak{c}\cap\mathfrak{h}^{\C}$ is a Borel subalgebra of $\mathfrak{h}^{\C}$.
Let $\Dcal$ be the left invariant co-CR structure induced by $\mathfrak{d}$\,, on $H$, and let $\e=\o|_H$\,.\\
\indent
Then the inclusion map from $\bigl(H,L\bigl(\Dcal,{\rm i}\,\e\bigr)\bigr)$ to $\bigl(G,L\bigl(\Cal,{\rm i}\,\o\bigr)\bigr)$ is holomorphic.\\
\indent
Fairly similar examples can be obtained by working with nilpotent Lie groups endowed with the generalized complex structures
of \cite{CavGua-nil}\,.
\end{exm}

\indent
The following facts are immediate consequences of the definitions.

\begin{rem}
1) A map between regular generalized almost complex manifolds is
holomorphic if and only if it is a co-CR Poisson morphism, with respect to the associated
almost co-CR and Poisson structures.\\
\indent
2) Let $\phi:(M,L_M)\to(N,L_N)$ be a diffeomorphism between generalized complex manifolds. Then $\phi$ is
holomorphic if and only if, in an open neighbourhood of each regular point of $M$, we have
$\phi_*\bigl(L_M\bigr)=L_N$\,, up to $B$-field transformations.\\
\indent
3) The composition of two holomorphic maps, between generalized (almost) complex manifolds is holomorphic.\\
\indent
4) Let $(M,L)$ be a generalized complex manifold. The $\overline{\partial}$ operator on functions \cite{Gua-thesis}
(see \cite{Hit-gholo_bundles}\,, and, also, \cite{Cav-ddbar}\,, \cite{Che-d_and_dbar}\,) is defined as follows.
If $f$ is a complex valued function on $M$ then $\overline{\partial}f$ is the $L$-component of $\dif\!f$ with respect to the decomposition
$T^{\C\!}M\oplus\bigl(T^{\C\!}M\bigr)^*=L\oplus\overline{L}$\,. Then $f$ is holomorphic if and only if $\overline{\partial}f=0$.
Note that, if $L=L(E,\ep)$ is regular then the holomorphic (local) functions on $(M,L)$
are just the co-CR holomorphic functions on $(M,E)$\,. Equivalently, if $E$ is locally defined by the submersion
$\phi:M\to(N,J)$ onto the complex manifold $(N,J)$ (that is, $E=\dif\!\phi^{-1}\bigl(T^{1,0}N\bigr)$\,) then, locally,
any holomorphic function on $(M,L)$ is the composition of $\phi$ followed by a holomorphic function on $(N,J)$\,.\\
\indent
5) Let $(M,J_M,L_M)$ and $(N,J_N,L_N)$ be complex manifolds endowed with holomorphic Poisson structures (see \cite{LGeStiXu}
for the notion of holomorphic Poisson structure, and \cite{CheStiXu} for a generalization); denote by $\widetilde{L}_M$
and $\widetilde{L}_N$ the generalized complex structures associated to $L_M$ and $L_N$\,, respectively (see \cite{Hit-gc_CMP}\,).
For any map $\phi:M\to N$, any two of the following assertions imply the third:\\
\indent
\quad(i) $\phi:(M,J_M,L_M)\to(N,J_N,L_N)$ is a holomorphic Poisson morphism;\\
\indent
\quad(ii) $\phi:(M,\widetilde{L}_M)\to(N,\widetilde{L}_N)$ is holomorphic;\\
\indent
\quad(iii) $\phi:(M,J_M)\to(N,J_N)$ is holomorphic and it maps the leaves of the holomorphic symplectic foliation
associated to $L_M$ into leaves of the holomorphic symplectic foliation associated to $L_N$\,.
\end{rem}

\indent
{}From Theorem \ref{thm:pbP} we obtain the following result.

\begin{prop} \label{prop:holo_Dirac}
Let $(M,L_M)$ and $(N,L_N)$ be regular real analytic generalized complex manifolds and let
$\phi:M\to N$ be a real analytic map.\\
\indent
If $\phi$ is holomorphic then, locally,
up to the complexification of a real analytic $B$-field tranformation, the complexification of $\phi$ descends
to a complex analytic Poisson morphism between canonical Poisson quotients.
\end{prop}

\indent
Let $L(E,{\rm i}\,\ep)$ be a generalized complex structure in normal form on a Riemannian manifold $(M,g)$\,.\\
\indent
Then $E$ is coisotropic (that is, $E^{\perp}$ is isotropic), with respect to $g$\,, if and only if $E\cap\overline{E}$
is locally defined by pseudo-horizontally conformal submersions onto complex manifolds (a map from a Riemannian manifold
to an almost complex manifold is \emph{pseudo-horizontally conformal} if it pulls back $(1,0)$-forms to isotropic forms).\\
\indent
Also, if $\ep^k$ has constant norm, with respect to $g$\,, where $\dim(E\cap\overline{E})=2k$\,, then
the leaves of $E\cap\overline{E}$ are minimal submanifolds of $(M,g)$\,.\\
\indent
Conversely, we have the following:

\begin{exm} \label{exm:gc_harmorphs}
Let $\phi:(M,g)\to(N,J)$ be a pseudo-horizontally conformal submersion from a Riemannian manifold
onto an almost complex manifold, with $\dim M=\dim N+2$\,.\\
\indent
Denote $\V={\rm ker}\dif\!\phi$\,, $\H=\V^{\perp}$ and let $\o$ be the volume form of $\V$. Also, let
$F$ be the unique skew-adjoint almost $f$-structure on $M$ such that ${\rm ker}F=\V$ and, with respect
to which, $\phi$ is co-CR holomorphic. Obviously, $F$ and $\o$ are compatible; denote by $L$ the
corresponding generalized almost complex structure in normal form.\\
\indent
{}From Proposition \ref{prop:int_normal_form} it follows that $L$ is integrable if and only if
$J$ is integrable, the fibres of $\phi$ are minimal and the integrability tensor of $\H$ is of type $(1,1)$\,;
note that, if $\dim M=4$ then this is equivalent to the condition that $\phi$ is a harmonic morphism
(see \cite{BaiWoo2}\,), where $N$ is endowed with the conformal structure with respect to which
$J$ is a Hermitian structure.\\
\indent
Moreover, any generalized complex structure, in normal form, on a Riemannian manifold such that the
corresponding $f$-structure is skew-adjoint, the associated Poisson structure has rank two and
its symplectic form has norm $1$ is, locally, obtained this way.\\
\indent
The pseudo-horizontally conformal submersions with geodesic fibres onto complex manifolds,
for which the integrability tensor of the horizontal distribution is of type $(1,1)$\,,
admit a twistorial description from which it follows that they abound on Riemannian manifolds
of constant curvature \cite{Pan-tm} (cf.\ \cite{BaiWoo2}\,).\\
\indent
Also, see \cite{AprAprBri-IJM} for a study of the harmonic pseudo-horizontally conformal submersions
with minimal fibres and \cite{BaiWoo2} for twistorial constructions of harmonic morphisms with two-dimensional fibres
on four-dimensional Riemannian manifolds.
\end{exm}

\section{Generalized K\"ahler manifolds} \label{section:gKm}

\indent
We start this section by recalling from \cite{Gua-thesis} a few facts on generalized K\"ahler manifolds.\\
\indent
A \emph{generalized (almost) K\"ahler manifold} is a manifold $M$ endowed with two generalized (almost)
complex structures such that the corresponding sections $\J_1$ and $\J_2$ of ${\rm End}(TM\oplus T^*M)$
commute and $\J_1\J_2$ is negative definite.\\
\indent
Any generalized almost K\"ahler structure $(L_1,L_2)$ on a manifold $M$ corresponds to a quadruple
$(g,b,J_+,J_-)$ where $g$ is a Riemannian metric, $b$ is a two-form and $J_{\pm}$ are almost Hermitian
structures on $(M,g)$\,. The (bijective) correspondence is given by $L_1=L^+\oplus L^-$\,, $L_2=L^+\oplus\overline{L^-}$,
where
$$L^{\pm}=\bigl\{X+(b\pm g)(X)\,|\,X\in V^{\pm}\bigr\}$$
with $V^{\pm}$ the eigenbundles of $J_{\pm}$ corresponding to ${\rm i}$\,.\\
\indent
According to \cite[Theorem 6.28]{Gua-thesis}\,, the following assertions are equivalent:\\
\indent
\quad(i) $L_1$ and $L_2$ are integrable.\\
\indent
\quad(ii) $L_+$ and $L_-$ are integrable.\\
\indent
\quad(iii) $J_{\pm}$ are integrable and parallel with respect to $\nabla^{\pm}=\nabla^g\pm\tfrac12g^{-1}h$\,,
where $\nabla^g$ is the Levi-Civita connection of $g$ and $h=\dif\!b$ (equivalently, $J_{\pm}$ are integrable
and $\dif^c_{\pm}\!\o_{\pm}=\mp h$\,, where $\o_{\pm}$ are the K\"ahler forms of $J_{\pm}$).\\
\indent
Now, if we (pointwisely) denote $E_j=\p(L_j)$\,, $(j=1,2)$\,, then $E_1=V^++V^-$
and $E_2=V^++\overline{V^-}$. Hence, $E_1^{\perp}=V^+\cap V^-$, $E_2^{\perp}=V^+\cap\overline{V^-}$ and,
therefore, $E_1$ and $E_2$ are coisotropic.

\begin{rem} \label{rem:assocF_holo-functions}
Let $(M,L_1,L_2)$ be a generalized K\"ahler manifold.\\
\indent
1) The (skew-adjoint) almost $f$-structures $F_j$ determined by $E_j$ and $E_j^{\perp}$ are integrable;
we call $F_j$ \emph{the $f$-structures of $L_j$}\,, $(j=1,2)$\,.\\
\indent
2) The holomorphic functions of $(M,L_1)$ and $(M,L_2)$ are the bi-holomorphic functions of $(M,J_+,J_-)$ and $(M,J_+,-J_-)$\,,
respectively.
\end{rem}

\indent
Let $\H^{\pm}={\rm ker}(J_+\mp J_-)$\,. Then $\H^+$ and $\H^-$ are orthogonal; this follows from
$\H^+=\bigl(V^+\cap V^-\bigr)\oplus\overline{\bigl(V^+\cap V^-\bigr)}$ and
$\H^-=\bigl(V^+\cap\overline{V^-}\bigr)\oplus\overline{\bigl(V^+\cap\overline{V^-}\bigr)}$\,.
Denote $\V=\bigl(\H^+\oplus\H^-\bigr)^{\perp}$.\\
\indent
Note that, $\H^+$, $\H^-$ and $\V$ are invariant under $J_+$ and $J_-$\,.
Also, $J_+-J_-$ and $J_++J_-$ are invertible on $\V$.

\begin{prop} \label{prop:regular}
The following assertions are equivalent:\\
\indent
{\rm (i)} $L_1$ and $L_2$ are regular.\\
\indent
{\rm (ii)} $\H^+$ and $\H^-$ are distributions on $M$.\\
\indent
{\rm (iii)} $\V$ is a distribution on $M$.
\end{prop}
\begin{proof}
The obvious relations
\begin{equation*}
\begin{split}
E_1&=\bigl(V^+\cap V^-\bigr)^{\perp}=\bigl(V^+\cap V^-\bigr)\oplus\H^-\oplus\V\;,\\
E_2&=\bigl(V^+\cap\overline{V^-}\bigr)^{\perp}=\bigl(V^+\cap\overline{V^-}\bigr)\oplus\H^+\oplus\V\\
\end{split}
\end{equation*}
imply
\begin{equation*}
\begin{split}
E_1\cap\overline{E_1}&=\H^-\oplus\V=\bigl(\H^+\bigr)^{\perp}\;,\\
E_2\cap\overline{E_2}&=\H^+\oplus\V=\bigl(\H^-\bigr)^{\perp}
\end{split}
\end{equation*}
which show that (i)$\Longleftrightarrow$(ii)\,.\\
\indent
Also, as the dimensions of $\H^+$ and $\H^-$ are upper semicontinuous functions on $M$, assertion (ii)
holds if and only if $\H^+\oplus\H^-\bigl(=\V^{\perp}\bigr)$ is a distribution on $M$.
\end{proof}

\indent
Next, we prove the following result.

\begin{thm} \label{thm:H+geod}
Let $(M,L_1,L_2)$ be a generalized K\"ahler manifold with $L_1$ regular.\\
\indent
Then the following assertions are equivalent:\\
\indent
\quad{\rm (i)} $\H^+$ is integrable.\\
\indent
\quad{\rm (ii)} $\H^+$ is geodesic.\\
\indent
Furthermore, if\/ {\rm (i)} or {\rm (ii)} holds then the leaves of $\H^+$, endowed with $(g,J_{\pm})$\,, are K\"ahler manifolds.\\
\indent
Also, if $\H^+$ is holomorphic, with respect to $J_+$ or $J_-$\,, then both {\rm (i)} and {\rm (ii)} hold.
\end{thm}

\indent
To prove Theorem \ref{thm:H+geod} we need some preparations.\\
\indent
Let $\H$ be a distribution on a Riemannian manifold $(M,g)$ endowed with a linear connection $\nabla$;
denote $\V=\H^{\perp}$.\\
\indent
The \emph{second fundamental form} of $\H$, with respect to $\nabla$,
is the $\V$-valued symmetric two-form $\Bh$ on $\H$ defined by
$\Bh(X,Y)=\tfrac12\,\V\bigl(\nabla_XY+\nabla_YX\bigr)$\,; then $\H$ is geodesic, with respect to $\nabla$,
if and only if $\Bh=0$ (cf.\ \cite{BaiWoo2}\,).\\
\indent
The next result follows from a straightforward calculation.

\begin{lem}[cf.\ \cite{Wat-Hsubm}\,] \label{lem:Watson}
Let $(M,g,J)$ be a Hermitian manifold endowed with a distribution $\H$ and a conformal connection $\nabla$ such that
$\nabla J=0$\,.\\
\indent
If\/ $\V$ is integrable and $J$-invariant then the following relation holds:
\begin{equation*}
2\,g\bigl(\Bh(JX,Y),V\bigr)+g\bigl(I^{\H}(X,Y),JV\bigr)=g\bigl(T(V,JX),Y\bigr)+g\bigl(T(V,X),JY\bigr)\;,
\end{equation*}
for any $X,Y\in\H$ and $V\in\V$, where $T$ is the torsion of $\nabla$ and $I^{\H}$ is the
integrability tensor of $\H$, defined by $I^{\H}(X,Y)=-\V[X,Y]$\,, for any sections $X$ and $Y$\,of $\H$.
\end{lem}

\indent
To prove Theorem \ref{thm:H+geod} we also need the following lemma.

\begin{lem} \label{lem:holoH}
Let $(M,J)$ be a complex manifold and let $\H$ be a holomorphic distribution on $(M,J)$\,.
The following assertions are equivalent:\\
\indent
{\rm (i)} $\H$ is integrable.\\
\indent
{\rm (ii)} $\H^{1,0}$ is a CR structure.
\end{lem}
\begin{proof}
This is obvious.
\end{proof}

\begin{proof}[Proof of Theorem \ref{thm:H+geod}]
We may assume that, also, $L_2$ is regular.\\
\indent
Obviously, the second fundamental form of $\H^+$, with respect to $\nabla^g$, is equal
to the second fundamental forms of $\H^+$, with respect to $\nabla^{\pm}$.\\
\indent
As $L_1$ and $L_2$ are integrable we have that $E_1$ and $E_2$ are integrable and, consequently,
$\H^+\oplus\V$ and $\H^-\oplus\V$ are integrable; in particular, the integrability tensor of $\H^+$ takes
values in $\V$. Furthermore, $\H^+\oplus\V$ and $\H^-\oplus\V$ are holomorphic with respect to both $J_+$ and $J_-$\,.\\
\indent
Now, by applying Lemma \ref{lem:Watson} to $\H=\H^+$ twice, with respect to $\nabla^+$ and $\nabla^-$, we quickly obtain
$$4\,g\bigl(B^{\H^+}\!(J_{\pm}X,Y),V\bigr)=-g\bigl(I^{\H^+}\!(X,Y),(J_++J_-)(V)\bigr)\;,$$
for any $X,Y\in\H^+$ and $V\in\H_-\oplus\V$. As $J_++J_-$ is invertible on $\V$, we obtain that
(i)$\Longleftrightarrow$(ii)\,.\\
\indent
If $\H^+$ is integrable then $(g,b,J_+,J_-)$ induces, by restriction,
a generalized K\"ahler structure on each leaf $L$ of $\H^+$ and $J_+=J_-$ on $L$\,.\\
\indent
If $\H^{\pm}$ is holomorphic with respect to $J_+$ or $J_-$ then $\H^+$ is integrable by Lemma \ref{lem:holoH}
and the fact that the eigenbundles of $J_{\pm}|_{\H^+}$ corresponding to ${\rm i}$ are equal to $V^+\cap V^-$ which
is integrable.\\
\end{proof}

\begin{rem} \label{rem:Riem_subm_descend}
1) Let $(M,L_1,L_2)$ be a generalized K\"ahler manifold with $L_1$ regular. If $\H^+$ is integrable then,
by Theorem \ref{thm:H+geod}\,, the co-CR structure associated to $L_1$ (that is, $E_1$) is, locally, given
by holomorphic Riemannian submersions from $(M,g,J_{\pm})$ onto K\"ahler manifolds $(P,h,J)$\,; in particular, the
leaves of $\H^+$, endowed with $(g,J_{\pm})$ can be, locally, identified with $(P,h,J)$\,.\\
\indent
2) If $(M,L_1^M,L_2^M)$ and $(N,L_1^N,L_2^N)$ are generalized K\"ahler manifolds with
$\H^+_M$ and $\H^+_N$ integrable distributions then any holomorphic map $\phi:(M,L_1^M)\to(N,L_1^N)$
descends, locally (with respect to the Riemannian submersions of Remark \ref{rem:Riem_subm_descend}(1)\,),
to a holomorphic map between K\"ahler manifolds.
\end{rem}

\indent
Let $(M_j,g_j,J_j)$ be K\"ahler manifolds, $(j=1,2)$\,. Then on $M_1\times M_2$ there are
two nonequivalent natural generalized K\"ahler structures: the \emph{first product} is just the
K\"ahler product structure whilst the \emph{second product} is given by
$L_1=L\bigl(T^{1,0}M_1\times TM_2,{\rm i}\,\o_2\bigr)$
and $L_2=L\bigl(T^{1,0}M_2\times TM_1,{\rm i}\,\o_1\bigr)$\,, where $\o_j$ are the K\"ahler forms
of $J_j$\,, $(j=1,2)$\,; see Section \ref{section:gK_tamed}\,, below, for the corresponding definitions in a more general setting.
Note that, both $L_1$ and $L_2$ are in normal form; moreover, the corresponding almost $f$-structures are skew-adjoint
(and, thus, unique with this property).\\
\indent
We end this section with the following consequence of Theorem \ref{thm:H+geod} (cf.\ \cite[Theorem A]{ApoGua}\,).

\begin{cor} \label{cor:H+-}
Any generalized K\"ahler manifold with $\V=0$ is, up to a unique $B$-field transformation, locally given by the
second product of two K\"ahler manifolds.
\end{cor}
\begin{proof}
Let $(M,L_1,L_2)$ be a generalized K\"ahler manifold with $\V=0$\,.
Then, Proposition \ref{prop:regular} implies that $\H^{\pm}$ are complementary orthogonal distributions on $M$.\\
\indent
As $L_1$ and $L_2$ are integrable, we have $\H^{\pm}$ integrable. Furthermore, by (the proof of) Theorem \ref{thm:H+geod}\,,
we have that $\H^{\pm}$ are geodesic foliations which are holomorphic with
respect to both $J_{\pm}$\,; moreover, $(g,J_{\pm})$ induce, by restriction, K\"ahler structures on their leaves.\\
\indent
If $L_2=L(E_2,\ep_2)$ then, from the definitions it follows that $\ep_2=(b-{\rm i}\,\e)|_{E_2}$\,,
where $\e$ is the two-form on $M$ characterised by $\iota_X\e=0$ if $X\in\H_-$ and $\e|_{\H^+}$ is the K\"ahler form
of $J_+|_{\H^+}$. As $(\Lie_X\!\e)(Y,Z)=0$ for any sections $X$ of $\H^-$ and $Y$, $Z$ of $\H^+$, and
$(\dif\!\ep_2)(X,Y,Z)=0$ for any $X,Y,Z\in E_2$\,, we obtain that
$(\dif\!b)(X,Y,\overline{Z})=0$ for any $X\in V^+\cap\H^-(=E_2\cap\H^-)$ and $Y,Z\in V^+\cap\H^+$.
Furthermore, from Lemma \ref{lem:Watson}\,, applied to $\H=\H^+$ with $J=J_+$ and
$\nabla=\nabla^+$, we obtain $(\dif\!b)(X,Y,Z)=0$ for any $X\in\H^-$ and $Y,Z\in V^+\cap\H^+$.\\
\indent
It follows that $\dif\!b=0$ and the proof is complete.
\end{proof}

\section{Tamed symplectic and generalized K\"ahler manifolds} \label{section:gK_tamed}

\indent
The following definition is fairly standard.

\begin{defn} \label{defn:tamed_symp}
A \emph{tamed almost symplectic manifold} is a manifold $M$ endowed with a nondegenerate
two-form $\ep$ and an almost complex structure $J$ such that $\ep(JX,X)>0$ for any nonzero $X\in TM$.\\
\indent
\indent
A \emph{tamed symplectic manifold} is a tamed almost symplectic manifold $(M,\ep,J)$ such that
$J$ and $\ep^{-1}J^*\ep$ are integrable and $\dif\!\ep=0$\,.
\end{defn}

\indent
Obviously, $(M,\ep,J)$ is a tamed symplectic manifold if and only if $\ep$ is a symplectic form,
$T^{1,0}M$ and $\bigl(T^{1,0}M\bigr)^{\perp_{\ep}}$ are integrable, and $\ep(JX,X)>0$\,, for any nonzero $X\in TM$.\\
\indent
The next result also appears, in a different form, in \cite{Gua-Pbranes}\,.

\begin{thm} \label{thm:gK_tamed}
Let $M$ be a manifold endowed with a nondegenerate two-form $\ep$ and an almost complex structure $J$;
denote $J_+=J$ and $J_-=-\ep^{-1}J^*\ep$\,. Let $g$ and\/ $b$ be the symmetric and skew-symmetric parts, respectively,
of\/ $\ep J$. \\
\indent
Then the following assertions are equivalent:\\
\indent
\quad{\rm (i)} $(M,\ep,J)$ is a tamed symplectic manifold.\\
\indent
\quad{\rm (ii)} $(g,b,J_+,J_-)$ defines a generalized K\"ahler structure such that $J_++J_-$ is invertible.\\
\indent
Moreover, up to a unique $B$-field transformation, any generalized K\"ahler structure, on $M$,
with $J_++J_-$ invertible is obtained this way from a tamed symplectic structure.
\end{thm}
\begin{proof}
Firstly, note that $\ep(J_+X,Y)=-\ep(X,J_-Y)$\,, for any $X,Y\in TM$. This implies that
\begin{equation} \label{e:gb_ep}
\begin{split}
g(X,Y)=\,&\tfrac12\,\ep\bigl((J_++J_-)(X),Y\bigr)\;,\\
b(X,Y)=\,&\tfrac12\,\ep\bigl((J_+-J_-)(X),Y\bigr)\;,
\end{split}
\end{equation}
for any $X,Y\in TM$.\\
\indent
Therefore $(M,\ep,J)$ is a tamed almost symplectic manifold if and only if the quadruple $(g,b,J_+,J_-)$ defines
a generalized almost K\"ahler manifold with $J_++J_-$ invertible.\\
\indent
Now, with respect to $J_{\pm}$\,, we have $\o_{\pm}=-\ep^{1,1}$, $b^{1,1}=0$ and
$b^{2,0}=\pm{\rm i}\,\ep^{2,0}$. It quickly follows
that if $J_{\pm}$ are integrable then $\dif\!\ep=0$ if and only if $\dif^c_{\pm}\!\o_{\pm}=\mp\dif\!b$\,.\\
\indent
We have thus proved that (i)$\Longleftrightarrow$(ii)\,.\\
\indent
Suppose that $(g,b,J_+,J_-)$ corresponds to the generalized K\"ahler structure $(L_1,L_2)$ on $M$. Then $J_++J_-$
is invertible if and only if $\p(L_2)=TM$. Hence, if $J_++J_-$ is invertible then, up to a unique $B$-field
transformation, we have $L_2=L(TM,{\rm i}\,\ep)$ for some  symplectic form $\ep$ on $M$ and, consequently,
\begin{equation}  \label{e:ep_gb}
\begin{split}
{\rm i}\,\ep(X-{\rm i}J_+X,Y)=\,&(b+g)(X-{\rm i}J_+X,Y)\;,\\
{\rm i}\,\ep(X+{\rm i}J_-X,Y)=\,&(b-g)(X+{\rm i}J_-X,Y)\;,
\end{split}
\end{equation}
for any $X,Y\in TM$. By using the fact that $J_++J_-$ is invertible, from \eqref{e:ep_gb} we quickly obtain
that $g$ and $b$ satisfy \eqref{e:gb_ep}\,. Together with the fact that $g$ and $b$ are symmetric and skew-symmetric,
respectively, this shows that $J_-=-\ep^{-1}J_+^*\ep$ and the proof follows.
\end{proof}

\indent
It is easy to rephrase Theorem \ref{thm:gK_tamed} so that to obtain the description of generalized K\"ahler manifolds
with $J_+-J_-$ invertible.\\
\indent
Let $(M,L^M_1,L^M_2)$ and $(N,L^N_1,L^N_2)$ be generalized K\"ahler manifolds corresponding to the tamed symplectic manifolds
$(M,\ep_M,J_M)$ and $(N,\ep_N,J_N)$\,, respectively. Then $(M\times N,L^M_1\times L^N_1,L^M_2\times L^N_2)$ and
$(M\times N,L^M_1\times L^N_2,L^M_2\times L^N_1)$ are called the \emph{first} and \emph{second product} of
$(M,L^M_1,L^M_2)$ and $(N,L^N_1,L^N_2)$\,, respectively; note that, the first product is the generalized K\"ahler manifold
corresponding to $(M\times N,\ep_M+\ep_N,J_M\times J_N)$\,.

\begin{cor} \label{cor:VH+}
Any generalized K\"ahler manifold with $\H^+$ a holomorphic foliation, with respect to $J_+$ and $J_-$\,, and $\H^-=0$ is,
up to a unique $B$-field transformation,
locally given by the first product of a K\"ahler manifold and a generalized K\"ahler manifold for which both $J_++J_-$ and $J_+-J_-$ are invertible.
\end{cor}
\begin{proof}
Let $(M,L_1,L_2)$ be a generalized K\"ahler manifold with $\H^+$ a distribution and $\H_-=0$\,. Then, by Theorem \ref{thm:gK_tamed}\,,
up to a unique $B$-field transformation, we have that $(M,L_1,L_2)$ corresponds to the tamed symplectic manifold $(M,\ep,J)$\,.\\
\indent
Thus, by \eqref{e:gb_ep}\,, we have $\iota_Xb=0$ for any $X\in\H^+$ and $\ep=\e+\ep'$ where $\e$ and $\ep'$ are
the two-forms on $M$ characterised by $\iota_X\e=0$\,, $(X\in\V)$\,, $\iota_X\ep'=0$\,, $(X\in\H^+)$\,, $\e=\o_+$ on $\H^+$\,, and
$\ep'=\ep$ on $\V$.\\
\indent
If, further, $\H^+$ is holomorphic with respect to $J_+$ and $J_-$\,, then, by Theorem \ref{thm:H+geod}\,, it is integrable, geodesic 
and its leaves endowed with $(g,J)$ are K\"ahler manifolds; in particular, $\dif\!\e=0$ on $\H^+$.
As, also, $\V$ is a holomorphic foliation, it quickly
follows that $(\Lie_X\!\e)(Y,Z)=0$ for any sections $X$ of $\V$ and $Y$, $Z$ of $\H^+$; consequently, $\dif\!\e=0$\,.\\
\indent
We have thus obtained $\dif\!\ep'=0$ which implies $(\Lie_X\!\ep')(Y,Z)=0$ for any sections $X$ of $\H^+$ and $Y$, $Z$ of $\V$.
Together with \eqref{e:gb_ep}\,, this gives $(\Lie_X\!b)(Y,Z)=0$ and $(\Lie_X\!g)(Y,Z)=0$ for any sections $X$ of $\H^+$ and
$Y$, $Z$ of $\V$; in particular, this shows that $\V$ is geodesic. The proof follows.
\end{proof}

\indent
Obviously, a result similar to Corollary \ref{cor:VH+} holds for any generalized K\"ahler manifold with $\H^+=0$
and $\H^-$ an integrable distribution.

\begin{cor} \label{cor:L2_normal}
Let $(M,L_1,L_2)$ be a generalized K\"ahler manifold such that $L_2$ is in normal form with respect to its
$f$-structure and the two-form $\ep$ on $M$.\\
\indent
Then, in a neighbourhood of each regular point of $L_1$\,, we have that $(M,L_1,L_2)$ is the second product
of a K\"ahler manifold and a generalized K\"ahler manifold determined by a tamed symplectic manifold.
\end{cor}
\begin{proof}
Assume $L_1$ regular. Define $\ep_{\pm}$ to be the  (complex linear) two-forms on $T^{1,0}_+M+T^{1,0}_-M$
such that $\ep_{\pm}=\ep$ on $T^{1,0}_{\pm}M$ and $\iota_X\ep_{\pm}=0$ if $X\in T^{1,0}_{\mp}M$.\\
\indent
Obviously, $\dif\!\ep_{\pm}=0$ on $T^{1,0}_{\pm}M$.
Also, from the fact that $\iota_X\ep_{\pm}=0$ if $X\in T^{1,0}_{\mp}M$ it quickly follows that if
$X_{\pm},\,Y_{\pm},\,Z_{\pm}\in T^{1,0}_{\pm}M$ then $\dif\!\ep_{\pm}(X_{\mp},Y_{\mp},Z_{\pm})=0$\,; together with the fact that
$\ep=\ep_++\ep_-$ on $T^{1,0}_+M+T^{1,0}_-M$, this implies that $\dif\!\ep_{\pm}(X_{\pm},Y_{\pm},Z_{\mp})=0$\,.
Thus, we have proved that $\dif\!\ep_{\pm}=0$ on $T^{1,0}_+M+T^{1,0}_-M$.\\
\indent
Therefore ${\rm ker}\,\ep_{\pm}=T^{1,0}_{\mp}\oplus\bigl(T^{0,1}_{\mp}\cap\H^-\bigr)$ is integrable which implies that $\H^-$ is
an antiholomorphic distribution on $(M,J_{\mp})$\,. Hence, by Lemma \ref{lem:holoH}\,, we have that $\H^-$ is integrable and
the proof follows from Theorem \ref{thm:H+geod} and the fact that ${\rm ker}\,\ep=\H^-$.
\end{proof}

\indent
Let $(M,\ep,J)$ be a tamed almost symplectic manifold. With the same notations as in Corollary \ref{cor:L2_normal}\,,
if $(M,L_1,L_2)$ is the generalized K\"ahler manifold determined by $(M,\ep,J)$ then, from \eqref{e:gb_ep}\,,
it follows that $L_1=L\bigl(T^{1,0}_+M+T^{1,0}_-M,\,{\rm i}\,\ep_+-{\rm i}\,\ep_-\bigr)$\,.

\begin{thm}[cf.\ \cite{Hit-gc_CMP}\,] \label{thm:holo_Poisson_tamed}
Let $(M,\ep,J)$ be a tamed almost symplectic manifold and let $(M,L_1,L_2)$ be the corresponding generalized almost K\"ahler manifold;
denote by $\r^{\pm}:T^{\C\!}M\to T^{1,0}_{\pm}M$ the projections.\\
\indent
If $(M,L_1,L_2)$ is generalized K\"ahler then $J_{\pm}$ are integrable and $\r^{\pm}_*(L_2)$ are holomorphic Poisson structures on $(M,J_{\pm})$\,,
respectively. Furthermore, the converse holds if also $J_+-J_-$ is invertible; moreover, in this case, if $(M,L_1,L_2)$ is generalized K\"ahler then
$\r^{\pm}_*(L_2)$ are holomorphic symplectic structures on $(M,J_{\pm})$\,, respectively.
\end{thm}
\begin{proof}
Assume, for simplicity, that $(M,\ep,J)$ is real analytic. Also, we may assume $L_1$ regular.
If $(M,L_1,L_2)$ is generalized K\"ahler then, by passing to the complexification of $(M,\ep,J)$\,, from Proposition \ref{prop:pfpb}
and the proof of Corollary \ref{cor:L2_normal} we obtain that $\r^{\pm}_*(L_2)$ are the canonical Poisson quotients of
$L\bigl(T^{0,1}_+M+T^{0,1}_-M,{\rm i}\,\overline{\ep_{\mp}}\bigr)$\,.\\
\indent
If $J_+\pm J_-$ are invertible and $J_{\pm}$ are integrable then $\r^{\pm}_*(L_2)$ are holomorphic Poisson structures on $(M,J_{\pm})$
if and only if $\dif\!\ep_{\pm}=0$\,.
\end{proof}

\indent
We call the $\r^{\pm}_*(L_2)$ of Theorem \ref{thm:holo_Poisson_tamed} the \emph{holomorphic Poisson structures associated to $(M,L_1,L_2)$}\,.

\begin{rem}[cf.\ \cite{Hit-gc_CMP}\,,\,\cite{Gua-Pbranes}\,] \label{rem:holo_Poisson_tamed}
1) Let $(M,L_1,L_2)$ be a generalized K\"ahler manifold with $J_++J_-$ invertible. Denote by $\e_{\pm}$ the (real) bivectors on $M$
which determine the holomorphic Poisson structures on $(M,J_{\pm})$\,, respectively, associated to $(M,L_1,L_2)$\,; that is,
with respect to $J_{\pm}$\,, we have $\e_{\pm}^{1,1}=0$ and the holomorphic bivectors corresponding to $\r^{\pm}_*(L_2)$
are $\e_{\pm}^{2,0}$, respectively.\\
\indent
It quickly follows that
$$\e_-=-\e_+=\tfrac12\bigl(J\ep^{-1}+\ep^{-1}J^*\bigr)=\tfrac12\bigl(J_+-J_-\bigr)\ep^{-1}=\tfrac14[J_+,J_-]g^{-1}\;,$$
where $(M,\ep,J)$ is the tamed symplectic manifold associated to $(M,L_1,L_2)$\,.\\
\indent
Hence, the symplectic foliation associated to $\e_+$ is given by $\V(={\rm im}(J_+-J_-)\,)$\,.\\
\indent
2) If the generalized almost K\"ahler structure $(L_1,L_2)$ on $M$ corresponds to the quadruple $(g,b,J_+,J_-)$ then $(L_2,L_1)$ corresponds
to $(g,b,J_+,-J_-)$\,. Assume that $(M,L_1,L_2)$ is a generalized K\"ahler manifold with $J_++J_-$ and $J_+-J_-$ invertible and let
$\e_+$ and $\e_+'$ be the bivectors which determine, as in (1)\,, the holomorphic symplectic structures
associated to $(M,L_1,L_2)$ and $(M,L_2,L_1)$\,, respectively. Then \eqref{e:Im-ep1_ep} implies that $\e_+'=-\e_+$\,.
\end{rem}

\indent
Next, we prove some results on holomorphic maps between generalized K\"ahler manifolds.

\begin{cor} \label{cor:holo_diffeo}
Let $(M,L_1,L_2)$ be a generalized almost K\"ahler manifold with $J_++J_-$ and $J_+-J_-$ invertible.\\
\indent
If $\phi:M\to M$ is a diffeomorphism then any two of the following assertions imply the third:\\
\indent
\quad{\rm (i)} $\phi:(M,L_1)\to(M,L_1)$ is holomorphic.\\
\indent
\quad{\rm (ii)} $\phi:(M,L_2)\to(M,L_2)$ is holomorphic.\\
\indent
\quad{\rm (iii)} $\bigl[\dif\!\phi\,,J_+J_-\bigr]=0$\,.
\end{cor}
\begin{proof}
Let $L=L\bigl(T^{1,0}_+M+T^{1,0}_-M,\ep_1\bigr)$\,. By using the first relation of \eqref{e:gb_ep}\,, we obtain
\begin{equation} \label{e:Im-ep1_ep}
({\rm Im}\,\ep_1)(J_+-J_-)=\ep(J_++J_-)\;,
\end{equation}
which, firstly, shows that if (iii) holds then (i)$\Longleftrightarrow$(ii)\,.\\
\indent
Furthermore, \eqref{e:Im-ep1_ep} implies that $\ep^{-1}({\rm Im}\,\ep_1)$ is skew-adjoint, with respect to $g$\,, and, consequently,
$\ep-{\rm Im}\,\ep_1$ is invertible. This fact together with \eqref{e:Im-ep1_ep} proves that (i)\,,\,(ii)$\Longrightarrow$(iii)\,.
\end{proof}

\begin{cor} \label{cor:Poisson_holo_map}
Let $(M,L^M_1,L^M_2)$ and $(N,L^N_1,L^N_2)$ be generalized K\"ahler manifolds, with $J^M_++J^M_-$ and $J^N_++J^N_-$ invertible,
and let $\phi:M\to N$ be a map.\\
\indent
{\rm (i)} If $\phi:(M,L^M_1)\to(N,L^N_1)$ and $\phi:(M,J^M_{\pm})\to(N,J^N_{\pm})$ are holomorphic
then $\phi$ is a holomorphic Poisson morphism between the corresponding
associated holomorphic Poisson manifolds; moreover, the converse holds if $\phi$ is an immersion.\\
\indent
{\rm (ii)} If $\phi:(M,L^M_2)\to(N,L^N_2)$ and, either, $\phi:(M,J^M_+)\to(N,J^N_+)$ or $\phi:(M,J^M_-)\to(N,J^N_-)$
are holomorphic maps then $\phi$ is a holomorphic Poisson morphism between the associated holomorphic Poisson structures.
\end{cor}
\begin{proof}
Assertion (i) follows from Proposition \ref{prop:pfpb} and the proof of Theorem \ref{thm:holo_Poisson_tamed}\,.\\
\indent
To prove (ii)\,, note that if $\phi:(M,L^M_2)\to(N,L^N_2)$ is holomorphic then $\phi:(M,J^M_+)\to(N,J^N_+)$ is holomorphic
if and only if $\phi:(M,J^M_-)\to(N,J^N_-)$ is holomorphic. The proof quickly follows from Remark \ref{rem:holo_Poisson_tamed}(1)\,.
\end{proof}

\indent
If $(g,J_{\pm})$ are K\"ahler structures on $M$ then $(g,0,J_+,J_-)$ corresponds to a generalized K\"ahler structure $(L_1,L_2)$
on $M$; furthermore, if $b$ is a closed two-form on $M$ then $(g,b,J_+,J_-)$ corresponds to $\bigl((\exp b)(L_1),(\exp b)(L_2)\bigr)$\,.

\begin{exm}[cf.\ \cite{Hit-gc_CMP}\,] \label{exm:hK_gK}
Let $(M,g,I,J,K)$ be a hyper-K\"ahler manifold. Denote by $\o_I$\,, $\o_J$\,, $\o_K$ the K\"ahler forms of
$I$, $J$, $K$, respectively, and let $\ep=-(\o_J+\o_K)$\,.\\
\indent
Then $(M,\ep,J)$ is a tamed symplectic manifold. The corresponding generalized K\"ahler structure $(L_1,L_2)$
is given by $(g,b,J_+,J_-)$\,, where $b=\o_I$\,, $J_+=J$ and $J_-=K$. Also, $L_1=L\bigl(T^{\C\!}M,2\,\o_I-{\rm i}(\o_J-\o_K)\bigr)$\,,
$L_2=L\bigl(T^{\C\!}M,-{\rm i}(\o_J+\o_K)\bigr)$ and $\ep_+=-{\rm i}(\o_I-{\rm i}\,\o_J)$\,, $\ep_-=-(\o_K-{\rm i}\,\o_I)$\,.
\end{exm}

\indent
We end with a generalization of Corollaries \ref{cor:H+-} and \ref{cor:VH+}\,.

\begin{thm} \label{thm:H+-int}
Let $(M,L_1,L_2)$ be a generalized K\"ahler manifold. Then the following assertions are equivalent:\\
\indent
{\rm (i)} $\H^+\oplus\H^-$ is a holomorphic foliation with respect to $J_+$ and $J_-$\,.\\
\indent
{\rm (ii)} Locally, up to a $B$-field transformation, $(M,L_1,L_2)$ is the first product of a generalized K\"ahler manifold
for which $J_+\pm J_-$ are invertible and the second product of two K\"ahler manifolds.
\end{thm}
\begin{proof}
By applying Lemma \ref{lem:Watson} to $\H=\H^+\oplus\H^-$ twice, with respect to $\nabla^+$ and $\nabla^-$, we obtain
\begin{equation} \label{e:H+-int}
\begin{split}
2g\bigl(B^{\H^+\oplus\H^-\!}(J_+X_+,X_-)&,V\bigr)+g\bigl(I^{\H^+\oplus\H^-\!}(X_+,X_-),J_+V\bigr)\\
&=(\dif\!b)(V,J_+X_+,X_-)+(\dif\!b)(V,X_+,J_+X_-)\;,\\
2g\bigl(B^{\H^+\oplus\H^-\!}(J_+X_+,X_-)&,V\bigr)+g\bigl(I^{\H^+\oplus\H^-\!}(X_+,X_-),J_-V\bigr)\\
&=-(\dif\!b)(V,J_+X_+,X_-)+(\dif\!b)(V,X_+,J_+X_-)\;,
\end{split}
\end{equation}
for any $X_{\pm}\in\H^{\pm}$ and $V\in\V$. Consequently, we, also, have
\begin{equation} \label{e:H+-int_}
g\bigr(I^{\H^+\oplus\H^-\!}(X_+,X_-),(J_+-J_-)(V)\bigl)=2\dif\!b(V,J_+X_+,X_-)\;,
\end{equation}
for any $X_{\pm}\in\H^{\pm}$ and $V\in\V$.\\
\indent
Suppose that (i) holds. Then, by \eqref{e:H+-int_}\,, we have $\dif\!b(V,X_+,X_-)=0$\,, for any $X_{\pm}\in\H^{\pm}$ and $V\in\V$.
Moreover, from Corollaries \ref{cor:H+-} and \ref{cor:VH+} it follows that $\dif\!b(X,Y,Z)=0$ if
$X,Y,Z\in\H^+\oplus\H^-$ or $X\in\H^{\pm}$ and $Y,Z\in\V\oplus\H^{\pm}$.\\
\indent
As $\dif(\dif\!b)=0$\,, this shows that $\dif\!b$ is basic with respect to $\H^+\oplus\H^-$. Hence, locally, there exists
a two-form $b'$, basic with respect to $\H^+\oplus\H^-$, such that $\dif\!b=\dif\!b'$.\\
\indent
Furthermore, from \eqref{e:H+-int} and \eqref{e:H+-int_} we obtain $B^{\H^+\oplus\H^-\!}(X_+,X_-)=0$\,, for any \mbox{$X_{\pm}\in\H^{\pm}$.}
Together with Theorem \ref{thm:H+geod} and Corollary \ref{cor:VH+}\,, this shows that $\V$ and $\H^+\oplus\H^-$ are geodesic foliations on $(M,g)$\,.\\
\indent
Thus, we have proved that $(M,L_1,L_2)$ is the first product of a generalized K\"ahler manifold with $\H^+=0=\H^-$
and a generalized K\"ahler manifold with $\V=0$\,. Hence, by Corollary \ref{cor:H+-}\,, assertion (ii) holds.\\
\indent
The implication (ii)$\Longrightarrow$(i) is trivial.
\end{proof}

\noindent
{\bf Acknowledgements.} We are grateful to the referee for suggesting the references
\cite{Cav-ddbar}\,, \cite{Che-d_and_dbar}\,, \cite{CheStiXu}\,, and \cite{LGeStiXu}\,,
to Henrique~Bursztyn for bringing to our attention \cite{BurRad} and \cite{BurWei-2005}\,,
and to Marco~Gualtieri for informing us about \cite{Gua-Pbranes}\,.


\begin{thebibliography}{10}

\bibitem{AboBoy}
M.~Abouzaid, M.~Boyarchenko, Local structure of generalized complex manifolds,
\textit{J. Symplectic Geom.}, {\bf 4} (2006) 43--62.

\bibitem{AleDav-GC}
D.~Alekseevsky, L.David, Invariant generalized complex and Kahler structures on Lie groups,
Preprint, I.M.A.R., Bucharest, 2008.

\bibitem{ApoGua}
V.~Apostolov, M.~Gualtieri, Generalized K\"ahler manifolds, commuting complex structures,
and split tangent bundles, \textit{Comm. Math. Phys.}, {\bf 271} (2007) 561--575.

\bibitem{AprAprBri-IJM}
M.~A.~Aprodu, M.~Aprodu, V.~Br\^\i nz\u anescu, A class of harmonic submersions and minimal
submanifolds, \textit{Internat. J. Math.}, {\bf 11}  (2000) 1177--1191.
	
\bibitem{BaiWoo2}
P.~Baird, J.~C.~Wood, \textit{Harmonic morphisms between Riemannian manifolds},
London Math. Soc. Monogr. (N.S.), no. 29, Oxford Univ. Press, Oxford, 2003.

\bibitem{BurRaw}
F.~E.~Burstall, J.~H.~Rawnsley, \textit{Twistor theory for Riemannian symmetric spaces.
With applications to harmonic maps of Riemann surfaces},
Lecture Notes in Mathematics, 1424, Springer-Verlag, Berlin, 1990.

\bibitem{BurRad}
H.~Bursztyn, O.~Radko, Gauge equivalence of Dirac structures and symplectic groupoids,
\textit{Ann. Inst. Fourier (Grenoble)}, {\bf 53} (2003) 309--337.

\bibitem{BurWei-2005}
H.~Bursztyn, A.~Weinstein, Poisson geometry and Morita equivalence, \textit{Poisson geometry,
deformation quantisation and group representations},  1--78,
London Math. Soc. Lecture Note Ser., 323, Cambridge Univ. Press, Cambridge, 2005.

\bibitem{Cav-ddbar}
G.~R.~Cavalcanti, The decomposition of forms and cohomology of generalized complex manifolds,
\textit{J. Geom. Phys.}, {\bf 57} (2006) 121--132.

\bibitem{CavGua-nil}
G.~R.~Cavalcanti, M.~Gualtieri, Generalized complex structures on nilmanifolds,
\textit{J. Symplectic Geom.}, {\bf 2} (2004) 393--410.

\bibitem{Che-d_and_dbar}
Z.~Chen, The operators $\partial$ and $\overline\partial$ of a generalized complex structure,
\textit{Pacific J. Math.}, {\bf 242} (2009) 53--69.

\bibitem{CheStiXu}
Z.~Chen, M.~Sti\'enon, P.~Xu, Geometry of Maurer-Cartan elements on complex manifolds,
\textit{Comm. Math. Phys.}, {\bf 297} (2010) 169--187.

\bibitem{Courant}
T.~J.~Courant, Dirac manifolds, \textit{Trans. Amer. Math. Soc.}, {\bf 319} (1990) 631--661.

\bibitem{Cra}
M.~Crainic, Generalized complex structures and Lie brackets, Preprint, (\href{http://arxiv.org/abs/math/0412097}
{arXiv:math/0412097}).

\bibitem{Gua-thesis}
M.~Gualtieri, Generalized complex geometry, D.\ Phil.\ Thesis, University of Oxford, 2003.

\bibitem{Gua-Pbranes}
M.~Gualtieri, Branes on Poisson varieties, \textit{The many facets of geometry}, 368--394, Oxford Univ. Press, Oxford, 2010.

\bibitem{Hit-gc_QJM}
N.~J.~Hitchin, Generalized Calabi-Yau manifolds, \textit{Q. J. Math.}, {\bf 54} (2003) 281--308.

\bibitem{Hit-gc_CMP}
N.~J.~Hitchin, Instantons, Poisson structures and Generalized K\"ahler Geometry,
\textit{Comm. Math. Phys.}, {\bf 265} (2006) 131-–164.

\bibitem{Hit-gholo_bundles}
N.~J.~Hitchin, Generalized holomorphic bundles and the $B$-field action, Preprint, (\href{http://arxiv.org/abs/1010.0207}
{arXiv:math/10100207}).

\bibitem{Hum}
J.~E.~Humphreys, \textit{Introduction to Lie algebras and representation theory},
third printing, revised, Graduate Texts in Mathematics, 9, Springer-Verlag, New York-Berlin, 1980.

\bibitem{KoNo}
S.~Kobayashi, K.~Nomizu, \textit{Foundations of differential geometry},
I, Wiley Classics Library (reprint of the 1963 original), Wiley-Interscience Publ., Wiley, New-York, 1996.

\bibitem{LGeStiXu}
C.~Laurent-Gengoux, M.~Sti\'enon, P.~Xu, Holomorphic Poisson manifolds and holomorphic Lie algebroids,
\textit{Int. Math. Res. Not. IMRN} {\bf 2008}, Art. ID rnn 088, 46 pp.

\bibitem{fq}
S.~Marchiafava, L.~Ornea, R.~Pantilie, Twistor Theory for CR-quaternionic manifolds and
related structures, Preprint, I.M.A.R., Bucharest, 2009, (\href{http://arxiv.org/abs/0905.1455}
{arXiv:09051455}).

\bibitem{Pan-tm}
R.~Pantilie, On a class of twistorial maps, \textit{Differential Geom. Appl.}, {\bf 26} (2008) 366--376.

\bibitem{PanWoo-d}
R.~Pantilie, J.~C.~Wood, Harmonic morphisms with one-dimensional fibres on Einstein manifolds,
\textit{Trans. Amer. Math. Soc.}, {\bf 354} (2002) 4229--4243.

\bibitem{Vai-Poisson_book}
I.~Vaisman, \textit{Lectures on the geometry of Poisson manifolds}, Progress in Mathematics, 118,
Birkh\"auser Verlag, Basel, 1994.

\bibitem{Vai-red_gc}
I.~Vaisman, Reduction and submanifolds of generalized complex structures,
\textit{Differential Geom. Appl.}, {\bf 25} (2008) 147--166.

\bibitem{Wat-Hsubm}
B.~Watson, Almost Hermitian submersions, \textit{J. Differential Geometry}, {\bf 11}  (1976) 147--165.

\bibitem{Wei-local_P}
A.~Weinstein, The local structure of Poisson manifolds, \textit{J. Differential Geometry}, {\bf 18} (1983) 523--557.



\end{thebibliography}
\end{document}